\documentclass[11pt]{article}

\usepackage{sectsty}
\usepackage{titlesec}
\allsectionsfont{\centering \normalfont\scshape \large}
\titlelabel{\thetitle.\quad}

\usepackage{enumerate}
\usepackage{amssymb}
\usepackage{amsthm}
\usepackage{amsmath}
\usepackage{float}

\usepackage{pstricks-add}
\usepackage{bbm}
\usepackage{amssymb}
\usepackage{amsthm}
\usepackage{graphicx}
\usepackage{amsmath}
\usepackage[bottom]{footmisc}
\usepackage{float}

\newcommand{\pref}[1]{(\ref{#1})}

\newcommand{\stirling}[2]{{ #1\brace #2}}

\newfloat{Table}{hbtp}{lop}

\def\QEQ{{%
    \setbox0\hbox{$\cup$}%
    \rlap{\hbox to \wd0{\hss $\text{ }\vee$ \hss}}\box0
}}

\def\REQ{{%
\setbox0\hbox{$\cup$}%
    \rlap{\hbox to \wd0{\hss $\text{ }\wedge$ \hss}}\box0
}}

\author{ Jorge Garza Vargas \footnote{This work was supported by CONACYT Grant 222668 and by CIMAT's (Center for Research in Mathematics, Guanajuato) scholarship for bachelor's thesis.}  \bigskip \\ UC Berkeley  \\ jgarzavargas@berkeley.edu
}

\date{}
\title{\Large{\textbf{The traffic distribution of the squared unimodular random matrix and a formula for the moments of its ESD}}}

\begin{document}
\maketitle

\newtheorem{theorem}{Theorem}[section]
\newtheorem{definition}{Definition}[section]
\newtheorem{conjecture}{Conjecture}[section]
\newtheorem{lemma}{Lemma}[section]
\newtheorem{corollary}{Corollary}[section]
\newtheorem{proposition}{Proposition}[section]
\newtheorem{observation}{Observation}[section]
\newtheorem{remark}{Remark}[section]
\newtheorem{result}{Result}[section]
\newtheorem{example}{Example}[section]

\begin{abstract}
The $k$-th moment of the mean empirical spectral distribution (ESD) of the squared unimodular random matrix of dimension $N$ can be expressed in the form $N^{-2k-1} Q_k(N)$, where $Q_k(x)$ is a polynomial of degree $k+1$ with integer coefficients. We use tools from traffic-free probability to express the coefficients of this polynomial in terms of the number of quotients, with a certain property, of some colored directed graphs. The obtained result disproves the formula conjectured in A. Lakshminarayan, Z. Puchala, K. Zyczkowski (2014).
\end{abstract}


\section{Introduction}

For every positive integer $N$, define the unimodular random matrix  $U_N$ as the random square matrix of dimension $N$ whose entries are i.i.d. complex random variables uniformly distributed on $\mathbb{T}$, the unitary circle in $\mathbb{C}$. Denote by $U_N^*$ the adjoint matrix of $U_N$ and define the squared unimodular random matrix by
$$\rho_N := \frac{1}{N^2} U_NU_N^*. $$  
The moments of the mean empirical spectral distribution of $\rho_N$, i.e. the set of values of the form $\mathbb{E}[\mathrm{tr}(\rho_N^k)]$, with $k$ a positive integer and $\mathrm{tr}$ denoting the normalized trace, were studied by Lakshminarayan, Puchala and Zyczkowski  in \cite{Arul} in relation with problems in quantum information theory. This moments can be written as $N^{-2k-1}Q_k(N)$, where $Q_k(x)$ is a polynomial with integer coefficients of degree $k+1$. Here, we will calculate the injective traffic distribution of the unimodular ensemble and use it to derive an explicit formula for the coefficients of $Q_k(x)$ in terms of certain combinatorial numbers described in Section 3. In particular, this result disproves the following conjecture. 

\begin{conjecture}
\label{conj}
(\cite{Arul}) For $k$ and $N$ positive integers, it holds that 
$$
\mathbb{E}[\mathrm{tr}(\rho_N^k)]  N^{-2k-1} \sum_{j=2}^{k+1} (-1)^{k-j+1}f_{k-1,k-j+1} N^{j},$$

where $f_{k,j} := \frac{1}{k+1} {{2k+2}\choose{k-j}} {{k+j}\choose {j}}$, are the elements of the Borel triangle. 
\end{conjecture}

We acknowledge that this research was motivated by the suggestion made by the first autor in \cite{Arul}, who communicated that the conjectured formula presented systematic deviations when doing simulations for values of $k$ beyond 5.

\section{The traffic distribution method}

The theory of traffic-free probability, introduced by Camille Male in \cite{Maletraffics},  was developed in the context of free probability and was motivated by  the problem of studying the asymptotic freeness of permutation invariant families of matrices.  

Formally, we will analyse the traffic distribution of $\rho_N$. This method can be simply thought as a generalization of the moment method. First we will need to define the notion of $K$-graph operation defined by  C\'ebron, Dahlqvist and Male in \cite{Cebron-Dahlqvist-Male}. This concept was originally introduced in \cite{Mingo-Speicher-Graphs} by Mingo and Speicher as the notion of \emph{graph of matrices}. 

\begin{definition}
Let $K$ be a nonnegative integer. A $K$-graph operation is a directed connected graph (which may have loops and multiple edges) with $K$ ordered edges, and two distinguished vertices, which are named \emph{input} and \emph{output}, and will be denoted by $v_{\mathrm{in}}$ and $v_{\mathrm{out}}$ respectively. The input and the output  may be the same.
\end{definition}

 Intuitively, each $K$-graph will work  as a template to define a multilinear operation that takes $K$ square matrices of the same size and returns another square matrix of the same dimension. Given a $K$-graph $g=(V, E, v_{\mathrm{in}}, v_{\mathrm{out}})$ and random matrices $A_1, \dots, A_K$ of dimension $N$, we will denote by $Z_g(A_1\otimes \cdots \otimes A_K)$ the resulting random matrix of substituting, for $r=1, \dots, K$, each random matrix $A_r$ in the $r$-th edge of $g$,  and define it as
\begin{equation}
\label{defKop}
Z_g(A_1\otimes \cdots \otimes A_K) (i,j) := \sum_{\substack{\kappa : V \to [N]\\ \kappa(v_{\mathrm{in}})=j, \kappa(v_{\mathrm{out}})=i}} \prod_{r=1}^K A_r(\kappa(w_r), \kappa(v_r)),  
\end{equation}
where the $r$-th edge of $g$ goes from the vertex $v_r$ to the vertex $w_r$.

For example, if $g$ has vertices $v_0, v_1, \dots, v_K$, with $v_{\mathrm{in}}=v_0$ and $v_{\mathrm{out}}= v_K$,  and edges $(v_0, v_1),$ $\dots, (v_{K-1}, v_K)$, then  $Z_g(A_1\otimes \cdots \otimes A_K)$ is the usual product of matrices $A_KA_{K-1} \cdots A_1$. The set of all $K$-graph operations, with $K$ running over all nonnegative, is denoted by $\mathcal{G}$.

The traffic distribution of a family of random matrices $\textbf{A}= (A_i)_{i\in I}$, with $A_i$ of fixed dimension $d$, is the set of values of the form 
$$\mathbb{E}[\mathrm{tr}(Z_g(A_{i_1}^{\varepsilon_1} \otimes \cdots \otimes A_{i_K}^{\varepsilon_K}))],$$
where $g$ runs over all elements in $\mathcal{G}$, and for each $K$-graph, the indices $i_1, \dots, i_K$ are non necessarily distinct elements in $I$, while the $\varepsilon_i$ are elements in $\{1, \ast\}$. Note that in particular, the traffic distribution contains all the information of the mixed moments of the family $\textbf{A}$.

Given a $K$-graph operation $g=(V, E, v_{\mathrm{in}}, v_{\mathrm{out}})$, denote by $G=(V',E')$ the directed graph obtained by identifying the input and the output of $g$, and forgetting the \emph{input} and \emph{output} labels. Then, from \pref{defKop} we can see that 
$$\mathrm{tr}(Z_g(A_1\otimes \cdots \otimes A_K)) = \frac{1}{N} \sum_{\kappa: V' \to [N]} \prod_{r=1}^K A_r(\kappa(w_r'), \kappa(v_r')),  $$
where the $v_r'$ and  $w_r'$ are now vertices of $G$, and $(v_r', w_r')$ are the directed edges. Note that with this modification the sum on the right of the latter equation runs over all possible functions on the set of vertices. This observation allows us to restrict our analysis to directed graphs with ordered edges. For $G$ a directed connected graph, possibly with loops and multiple edges, with set of vertices $V$ and with $K$ ordered edges, $(v_1, w_1), \dots, (v_r, w_r)$, we denote 
$$\tau[G(A_1, \dots, A_K)] :=\frac{1}{N} \mathbb{E} \left[ \sum_{\kappa: V \to [N]} \prod_{r=1}^K A_r (\kappa(w_r), \kappa(v_r)) \right].$$
The function $\tau$ is called the traffic state in analogy with the theory of non-commutative probability. Note that the traffic distribution of a family of random matrices $\textbf{A}$ can also be thought as the set of values that $\tau$ takes when evaluated on all possible directed graphs in which elements of $\textbf{A}$ or $\textbf{A}^{\ast}$ are put in their edges in all possible ways. The injective version of the traffic state, introduced in \cite{Maletraffics},  is denoted by $\tau^0[\cdot]$ and is defined by
$$\tau^{0}[G(A_1, \dots, A_K)]:=\frac{1}{N} \mathbb{E} \left[ \sum_{\substack{\kappa :V \to [N]\\ \kappa \text{ injective}}} \prod_{r=1}^K A_r(\kappa(w_r), \kappa(v_r)) \right].$$
A direct computation proves the following relation between the traffic state and the injective traffic state.
\begin{lemma}
If $G=(V, E)$ is a directed graph with $K$ ordered edges and $A_1, \dots, A_K$ are $N\times N$ random matrices, then 
$$\tau [G(A_1, \dots, A_K)] = \sum_{\pi \in \mathcal{P}(V)}  \tau^0[G^{\pi}(A_1, \dots, A_K)],$$
where $\mathcal{P}(V)$ denotes the set of partitions of $V$ and for each $\pi \in \mathcal{P}(V)$,  $G^\pi$ denotes the quotient graph induced by $\pi$, that is, $G^\pi$ is the graph obtained by taking $G$ and identifying all vertices that are in a same block in $\pi$, without erasing any edges.
\end{lemma}

In particular, note that if $A$ is a random matrix, $k$ is a positive integer, and $C_k$ is the directed cycle of $k$ vertices, we have that 
\begin{equation}
\label{decoftrace}
\mathbb{E}[\mathrm{tr}(A^k)] = \tau[C_k(A)]= \sum_{\pi \in \mathcal{P}(k)} \tau^0[C_k^{\pi}(A)],
\end{equation}
where $C_k(A)$ is a shorthand notation for $C_k(A, \dots, A)$.

In Section 2.3 of \cite{Cebron-Dahlqvist-Male}, it is noted that, using the M\"obius inversion formula, the traffic state can be retrieved from the injective traffic state. Hence, these two functionals posses essentially the same information, nervetheless, when working with random matrices, it is usually easier to compute the values  of injective version of the traffic state, which is known as finding the injective traffic distribution of the given family of random matrices.   

\section{The formula}

Once we have equation \pref{decoftrace}, to find the wanted formula it will be enough to study for every $k$ and every $\pi \in \mathcal{P}(k)$, the value of $\tau^0[C_k^{\pi}(\rho_N)]$. To do this we will study the injective  traffic distribution of the unitary ensemble $U_N$. The latter is possible do to the fact that the mixed moments of the uniform probability measure on $\mathbb{T}$ have a nice formula. Explicitly, for $k$ and $l$ non negative integers
\begin{equation}
\label{moments}
\frac{1}{2\pi}\int_{\mathbb{T}} z^{k} \overline{z}^l dz = \delta_{kl}.
\end{equation}

Let $G=(V,E)$ be a directed connected graph. To every edge $e\in E$ we will assign either $U_N$ or $U_N^*$. Fix a labeling of this sort and denote it by $G(U_N, U_N^*)$.   Given $G(U_N, U_N^*)$, let $c:E\to \{\text{red},\text{blue}\}$, be the coloring of $G$ such that  $c(e)=\text{red}$ if $e$ is labeled with $U_N$ and $c(e)=\text{blue}$ if $e$  is labeled with $U_N^*$. Let $E_1=\{e\in E: c(e)=\text{red}\}$ and $E_2=\{e\in E: c(e)=\text{blue}\}$. Denote by $\hat{G}(U_N, U_N^*)$ the resulting colored graph.

\begin{figure}[htbp]
\begin{center}
\psset{xunit=1.2cm,yunit=1.2cm,algebraic=true,dimen=middle,dotstyle=o,dotsize=5pt 0,linewidth=1.6pt,arrowsize=3pt 2,arrowinset=0.25}
\begin{pspicture*}(-1.150069508506187,1.515490504884702)(10.55807888747329,3.6520927006487005)
\pscircle[linewidth=2.pt](-0.5754069700607644,2.1485782891127525){0.1}
\pscircle[linewidth=2.pt](1.4245930299392353,2.1485782891127525){0.1}
\pscircle[linewidth=2.pt](3.424593029939236,2.1485782891127525){0.1}
\pscircle[linewidth=2.pt](2.4245930299392353,3.1485782891127525){0.1}
\psline[linewidth=1.2pt]{->}(-0.5754069700607644,2.248578289112753)(1.3645876831703316,2.247219736684062)
\psline[linewidth=1.2pt]{->}(1.4245930299392353,2.048578289112752)(-0.5134963481544288,2.048407911796803)
\psline[linewidth=1.2pt]{->}(1.4921040854464376,2.2223499461654685)(2.33789951264774,3.0647350238863167)
\psline[linewidth=1.2pt]{->}(3.3245930299392357,2.1485782891127525)(1.547553265520604,2.1493770914428523)
\parametricplot[linewidth=1.2pt]{-2.8174164875305885}{2.8076043480745936}{1.*0.3067301387810374*cos(t)+0.*0.3067301387810374*sin(t)+3.719196759830717|0.*0.3067301387810374*cos(t)+1.*0.3067301387810374*sin(t)+2.146354561704604}
\pscircle[linewidth=2.pt](5.224593029939234,2.1485782891127525){0.1}
\pscircle[linewidth=2.pt](7.224593029939236,2.1485782891127525){0.1}
\pscircle[linewidth=2.pt](9.224593029939237,2.1485782891127525){0.1}
\pscircle[linewidth=2.pt](8.224593029939236,3.1485782891127525){0.1}
\psline[linewidth=1.2pt,linestyle=dashed,linestyle=dashed,dash=3pt 3pt,linecolor=blue]{->}(5.224593029939234,2.248578289112753)(7.164587683170332,2.247219736684062)
\psline[linewidth=1.2pt,linecolor=red]{->}(7.224593029939236,2.048578289112752)(5.286503651845571,2.048407911796803)
\psline[linewidth=1.2pt,linecolor=red]{->}(7.292104085446439,2.222349946165467)(8.137899512647742,3.0647350238863167)
\psline[linewidth=1.2pt,linestyle=dashed,linestyle=dashed,dash=3pt 3pt,linecolor=blue]{->}(9.124593029939238,2.1485782891127525)(7.347553265520601,2.1493770914428523)
\parametricplot[linewidth=1.2pt,linestyle=dashed,dash=3pt 3pt,linecolor=blue]{-2.8174164875305885}{2.8076043480745936}{1.*0.3067301387810374*cos(t)+0.*0.3067301387810374*sin(t)+9.519196759830718|0.*0.3067301387810374*cos(t)+1.*0.3067301387810374*sin(t)+2.146354561704604}
\rput[tl](0.37145023696211354,2.680909884392338){$U_N^*$}
\rput[tl](2.3080524327261102,2.5622097624054487){$U_N^*$}
\rput[tl](4.007990337833158,2.519046081682944){$U_N^*$}
\rput[tl](0.3174956360589823,1.866183542772674){$U_N$}
\rput[tl](1.3908242173728795,2.883055649449873){$U_N$}
\psline[linewidth=1.2pt]{->}(3.5345462863189576,2.3912781838167176)(3.4299511277512162,2.2484346398764243)
\psline[linewidth=1.2pt,linecolor=blue]{->}(9.338315969687383,2.3940752063706405)(9.229951127751216,2.2484346398764115)
\end{pspicture*}
\end{center}
\caption{ A labeled graph with its respective colored graph (blue arrows are displayed with a dotted style).}
\end{figure}
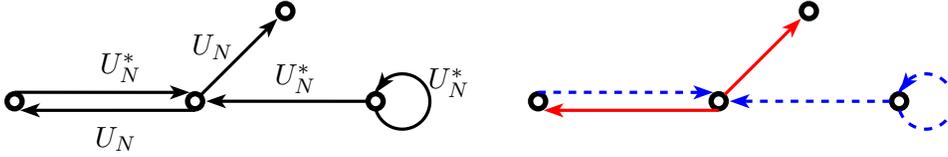

    Denote by $\tau^0[G(U_N, U_N^*)]$  the value of the injective traffic state evaluated on $G$ with the fixed choice of assignments of $U_N$ and $U_N^*$ to the edges. We have
\begin{equation}
\label{traffdistU}
\tau^0[G(U_N, U_N^*)]= \frac{1}{N} \sum_{\substack{\kappa :V \to [N]\\ \kappa \text{ injective}}} \mathbb{E}\left[ \prod_{(u,v)\in E_1} U_N(\kappa(v), \kappa(u)) \prod_{(r, s)\in E_2} \overline{U_N(\kappa(r), \kappa(s))} \right].
\end{equation}

Note that since the expectation of the product of independent random variables factorizes, by \pref{moments} each term in the sum in the right of \pref{traffdistU} is either 0 or 1. Now we will characterize those terms that do not vanish in terms of the structure of $G$ and the coloring $c$, for this we need the next definition.

\begin{definition} A connected directed graph colored in red and blue is called a \textbf{double directed colored graph  (d.d.c.g.)}, if for every $u, v \in V$, not necessarily distinct, the number of red edges going from $u$ to $v$ is equal to the number of blue edges going from $v$ to $u$.  We will denote the family of d.d.c.g's by $\mathcal{D}$. 
\end{definition}

\begin{figure}[htbp]
\begin{center}
\newrgbcolor{qqzzqq}{0. 0.6 0.}
\newrgbcolor{qqqqcc}{0. 0. 0.8}
\psset{xunit=1.2cm,yunit=1.2cm,algebraic=true,dimen=middle,dotstyle=o,dotsize=5pt 0,linewidth=0.8pt,arrowsize=3pt 2,arrowinset=0.25}
\begin{pspicture*}(2.3558370121937435,-1.3615470970400518)(16.82816463602887,2.791946032670507)
\pscircle[linewidth=2.pt](-2.,2.4641016151377557){0.09836262875318949}
\pscircle[linewidth=2.pt,linecolor=qqzzqq](0.,2.4641016151377553){0.09670872187124512}
\pscircle[linewidth=2.pt](1.,0.7320508075688774){0.09759296348517955}
\pscircle[linewidth=2.pt,linecolor=qqzzqq](0.,-1.){0.1}
\pscircle[linewidth=2.pt](-2.,-1.){0.1}
\pscircle[linewidth=2.pt,linecolor=qqzzqq](-3.,0.732050807568879){0.1003207474366654}
\pscircle[linewidth=2.pt](4.927781356839965,2.5042230835599977){0.09836262875318856}
\pscircle[linewidth=2.pt](7.927781356839964,0.7721722759911196){0.09759296348518111}
\pscircle[linewidth=2.pt](4.927781356839964,-0.9598785315777578){0.1}
\pscircle[linewidth=2.pt](6.,0.7320508075688774){0.09912084741145608}
\psline[linewidth=1.2pt,linestyle=dashed,dash=3pt 3pt, linecolor=blue]{->}(-1.90249676408703,2.4511275748365398)(-0.09612428055020428,2.453485609548297)
\psline[linewidth=1.2pt,linecolor=red]{->}(0.06988202916291382,2.397250445469278)(0.9383275722856385,0.8076874268854257)
\psline[linewidth=1.2pt,linestyle=dashed,dash=3pt 3pt, linecolor=blue]{->}(0.964312907900916,0.6412167942295391)(0.054255243755362725,-0.915997806427176)
\psline[linewidth=1.2pt,linecolor=red]{->}(-0.099630188238299,-1.0085921820046546)(-1.90057671059922,-1.010724249378336)
\psline[linewidth=1.2pt,linestyle=dashed,dash=3pt 3pt, linecolor=blue]{->}(-2.093597736422374,-0.9647939815286102)(-2.9599481060278943,0.6400720156899986)
\psline[linewidth=1.2pt,linecolor=red]{->}(-2.9792010327377914,0.8301918047468004)(-2.0566766644061754,2.3837090629514526)
\psline[linewidth=1.2pt,linestyle=dashed,dash=3pt 3pt, linecolor=blue]{->}(5.019697494226831,2.53924636757387)(6.06723914849081,0.8217124337338908)
\psline[linewidth=1.2pt,linecolor=red]{->}(5.904183041215201,0.7574290029498949)(4.873209882867254,2.391212478368995)
\psline[linewidth=1.2pt,linestyle=dashed,dash=3pt 3pt, linecolor=blue]{->}(4.844919183283536,-0.9038976402530898)(5.868198257539013,0.6938199575398962)
\psline[linewidth=1.2pt,linecolor=red]{->}(6.053035634463368,0.6483121434486915)(5.056371421970577,-0.9661164771977976)
\psline[linewidth=1.2pt,linecolor=red]{->}(6.06723914849081,0.8217124337338907)(7.827467491983496,0.8570923937436039)
\psline[linewidth=1.2pt,linestyle=dashed,dash=3pt 3pt, linecolor=blue]{->}(7.931475806877554,0.6746492655743997)(6.053035634463368,0.6483121434486915)
\pscircle[linewidth=2.pt](10.,0.7){0.1}
\pscircle[linewidth=2.pt](12.,0.7){0.1}
\pscircle[linewidth=2.pt](14.,0.7){0.1}
\psline[linewidth=1.2pt,linestyle=dashed,dash=3pt 3pt,linecolor=qqqqcc]{->}(14.,0.8)(12.070186659442916,0.7983671480442277)
\psline[linewidth=1.2pt, linestyle=dashed,dash=3pt 3pt, linecolor=qqqqcc]{->}(14.,0.6)(12.070186659442916,0.6012721613954357)
\parametricplot[linewidth=1.2pt,linecolor=red]{0.7853981633974483}{2.356194490192345}{1.*1.4142135623730951*cos(t)+0.*1.4142135623730951*sin(t)+13.|0.*1.4142135623730951*cos(t)+1.*1.4142135623730951*sin(t)+-0.2}
\parametricplot[linewidth=1.2pt,linecolor=red]{3.9269908169872414}{5.497787143782138}{1.*1.4142135623730951*cos(t)+0.*1.4142135623730951*sin(t)+13.|0.*1.4142135623730951*cos(t)+1.*1.4142135623730951*sin(t)+1.6}
\psline[linecolor=red]{->}(13.941136238280428,0.8555863683258547)(14.,0.8)
\psline[linecolor=red]{->}(13.950430688777374,0.5527743768269571)(14.,0.6)
\psline[linewidth=1.2pt,linecolor=red]{->}(10.,0.8)(11.9090383708087,0.8013567315269511)
\psline[linewidth=1.2pt,linestyle=dashed,dash=3pt 3pt,linecolor=qqqqcc]{->}(10.,0.6)(11.91357583949248,0.6017081094406034)
\psline[linecolor=red]{->}(13.83224735307488,0.9434003425309341)(14.,0.8)
\psline[linecolor=red]{->}(13.789707259154389,0.42681525545255083)(14.,0.6)
\rput[tl](9.819144979642351,1.0939794296359572){$u$}
\rput[tl](11.809360437628646,1.0939794296359572){$v$}
\end{pspicture*}
\caption{In the previous figure blue arrows are displayed with a dotted style. The graph on the left is a d.d.c.g. On the other hand, the graph on the right is not a d.d.c.g. because there is one red edge going from $u$ to $v$ but no blue edge going from $v$ to $u$.}
\end{center}
\end{figure}
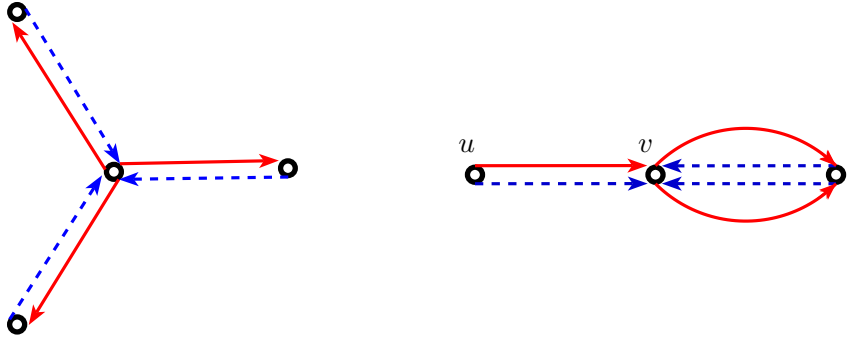

\begin{lemma} Let $\kappa: V \to [N]$ be injective. Then
\begin{equation}
\label{eqexpectation}
\mathbb{E}\left[ \prod_{(u,v)\in E_1} U_N(\kappa(v), \kappa(u)) \prod_{(r, s)\in E_2} \overline{U_N(\kappa(r), \kappa(s))} \right]\neq 0,
\end{equation}
 if and only if the resulting colored graph $\hat{G}(U_N, U_N^*)=(V, E, c)$ is in $\mathcal{D}.$ 
\end{lemma}

\begin{proof}
Since $\kappa$ is injective, for any $\{u, v\}$ and $\{r,s\}$ different sets of vertices, the sets $\{\kappa(u),\kappa(v)\}$ and $\{\kappa(r), \kappa(s)\}$ are different. In this case, given that the entries of $U_N$ are independent, $U_N^{\varepsilon_1}(\kappa(u),\kappa(v))$ and $U_N^{\varepsilon_2}(\kappa(r), \kappa(s))$ are independent random variables, for any $\varepsilon_1, \varepsilon_2 \in \{1, \ast\}$.

 Hence, edges with different endpoints contribute with independent random variables in the product in \pref{eqexpectation}.  Moreover, since $U_N(i,j)$ and $U_N^*(k,l)$ are only dependent if $i=l$ and $j=k$, we will have that edges with the same endpoints contribute with dependent random variables only if they go in the same direction and are of the same color, or if they go in opposite directions and are of different colors. Hence, the left side of \pref{eqexpectation} factorizes in the following way
$$\prod_{u , v\in V}\mathbb{E}\left[ U_N(\kappa(v), \kappa(u))^{d_1(u,v)} \overline{U_N}(\kappa(u), \kappa(v))^{d_2(v,u)}\right], $$
where $d_1(u,v)$ denotes the number of red edges going from $u$ to $v$ and $d_2(v,u)$ the number of blue edges going from $v$ to $u$.   By \pref{moments}, the latter product is different from zero only if for every $u$ and $v$, not necessarily distinct, it is satisfied that $d_1(u,v) = d_2(v,u)$, in which case the product equals to $1$.  Now note that the condition $d_1(u,v) = d_2(v,u)$ for every $u, v \in V$ translates into the condition $\hat{G}(U_N, U_N^*) \in \mathcal{D}$. 
\end{proof}
 
Combining \pref{traffdistU} with the previous lemma one sees that  
\begin{equation*}
\tau^0[G(U_N, U_N^*)]= \frac{1}{N} \sum_{\substack{\kappa :V \to [N]\\ \kappa \text{ injective}}}  \mathbbm{1}_{\{ \hat{G}(U_N, U_N^*) \in \mathcal{D}\}}= \begin{cases} \frac{(N)_{|V|}}{N} & \text{if } \hat{G}(U_N, U_N^*) \in \mathcal{D},  \\
0 & \text{otherwise.}
\end{cases}
\end{equation*} 
where, for $j$ a positive integer, $(N)_j:= N(N-1)\cdots(N-j+1)$ denotes the Pochhammer symbol.

Now note that if $C_k$ is the directed cycle with $k$ edges and $\rho_N$ is the squared unimodular matrix, then $\tau[C_k(\rho_N)] = N^{-2k} \tau[C_{2k}(U_N, U_N^*)]$, where the labeling $C_{2k}(U_N, U_N^*)$ alternates between $U_N$ and $U_N^*$, i.e. the resulting colored graph $\widehat{C_{2k}}(U_N, U_N^*)$ alternates between red and blue. From the characterization of the injective traffic distribution obtained before we get that 
$$\tau[C_{2k}(U_N, U_N^*)] = \frac{1}{N} \sum_{\pi \in \mathcal{P}(2k) } \mathbbm{1}_{\{\widehat{C_{2k}^{\pi}}(U_N, U_N^*)\in \mathcal{D}\}} (N)_{|\pi|}.$$
Denote by $\mathcal{F}(2k, j)$ the number of partitions $\pi \in \mathcal{P}(2k)$ with $j$ blocks and such that $\widehat{C_{2k}^{\pi}}\in \mathcal{D}$. Since $C_{2k}$ has $2k$ edges, if $|\pi| > k+1$, then $\widehat{C_{2k}^{\pi}}$ can not be a d.d.c.g., so $\mathcal{F}(2k, j) =0$ for all $j > k+1$. From what we have developed until now it follows that
\begin{equation}
\label{formulapochhammer}
\mathbb{E}[\mathrm{tr}(\rho_N^k)] = \tau[C_{2k}(U_N, U_N^*)] = N^{-2k-1} \sum_{j=1}^{k+1} \mathcal{F}(2k, j)(N)_j.
\end{equation}

\subsection{The Pochhammer symbols and the change of basis}

 The formula that we have provided for $N^{2k+1}\mathbb{E}[\mathrm{tr}(\rho_N^k)]$ is clearly a polynomial in $N$, nevertheless, it is not clear what its coefficients are, which, for example, makes it difficult to compare to the conjecture made in \cite{Arul}.
 
In traffic-free probability, the Pochhammer symbols appear frequently, so  it is convenient to be able to go from the Pochhammer symbols to the monomials of the form $N^k$ and vice versa.  More precisely, the family of polynomials $\{(x)_{j}\}_{j=0}^{\infty}$, where $(x)_0:=1$, form a basis for the space of polynomials. Below we will provide the formulas to change from the Pochhammer symbols basis to the usual basis $\{x^j\}_{j=1}^\infty$ and to do the reverse process as well.  

We will denote the $k$-th symmetric elementary polynomial in $n$ variables by $e_k(x_1, \dots, x_n)$, i.e.
$$e_k(x_1, \dots, x_n) := \sum_{1\leq j_1 < j_2 < \cdots < j_k\leq n} x_{j_1}\cdots x_{j_k},$$
when $1\leq k\leq n$, and define $e_0(x_1, \dots, x_n):=1$ and $e_k(x_1, \dots, x_n) = 0$ for $k > n$. Using the Vieta relations, for $k>0$ we get that  
$$(x)_k = \sum_{j=1}^k (-1)^{k-j} e_{k-j}(1, \dots, k-1)x^j,$$
which yields the following result.
 \bigskip
 
\begin{lemma}
\label{cambioPoch-Usual}
If $\sum_{j=1}^n a_j x^j = \sum_{j=1}^n b_j(x)_j$ then $a_j = \sum_{k=j}^n (-1)^{k-j}e_{k-j}(1, \dots, k-1) b_k$ for every $j=1, \dots, n$. 
\end{lemma}

To change from the usual basis to the Pochhammer basis we will use the following well known combinatorial identity
$$N^k = \sum_{j=1}^k \stirling{k}{j} (N)_j,$$
where the numbers $\stirling{k}{j}$ are the Stirling numbers of the second kind, which denote the number of partitions with $j$ blocks of a set with $k$ elements. This equality can easily be proven by a double counting argument for every positive integer $N$. Since it is valid for an infinite number of values it also holds as an equality of polynomials, yielding the following result. 
\bigskip

\begin{lemma}
\label{cambioUsual-Poch}
If $\sum_{j=1}^n a_j x^j = \sum_{j=1}^n b_j(x)_j$ then $b_j = \sum_{k=j}^n \stirling{k}{j}a_k$, for every $j=1, \dots, n$. 
\end{lemma}
\bigskip

If we assume that the Conjecture \ref{conj} is true, we get the following equality for every positive integers $k$ and $N$ 
$$\sum_{j=1}^{k+1} \mathcal{F}(2k, j)(N)_j = \sum_{j=2}^{k+1} (-1)^{k-j+1}f_{k-1,k-j+1} N^{j}, $$
using Lemma \ref{cambioUsual-Poch} we get that the latter is equivalent to 
\begin{equation}
\label{conjforG}
\mathcal{F}(2k, j) = \sum_{r=j}^{k+1} (-1)^{k-r+1} \stirling{r}{j}f_{k-1, k-r+1}.
\end{equation}
The above formula turns out to be true for all $j$ when $k\leq 5$. Making use of a computer to do the calculations, for $k=6$ and $j=3$, we obtained that the right hand side is 10988 while $\mathcal{F}(6,3)$ is actually 11000. This disproves the conjecture.

On the other hand, we can use Lemma \ref{cambioPoch-Usual} to obtain a formula in the usual basis, yielding the following theorem. 

\begin{theorem}
\label{theorem}
For every $k$ and $N$ positive integers we have that $\mathbb{E}[\mathrm{tr}(\rho_N^k)] = N^{-2k-1} \sum_{j=1}^{k+1} a_j N^j$, where
\begin{equation}
\label{formulafinal}
a_j= \sum_{r=j}^{k+1} (-1)^{r-j}e_{r-j}(1, \dots, k) \mathcal{F}(2k, j). 
\end{equation}
\end{theorem}

\section{Final remarks}

We used a computer to calculate the values of $\mathcal{F}(2k, j)$. Since the conjecture was based in the first four values of $k$, it is not a surprise that equation \pref{conjforG} holds for all $j$ when $k\leq 4$. However, it was surprising that it also turned out to be true when $k=5$,  but not true in general for greater values of $k$. It is also worth remarking how small is the error of the proposed formula for $\mathcal{F}(2k, j)$ when $k$ is small and that in fact it is always true for the cases $\mathcal{F}(2k, 1), \mathcal{F}(2k, 2)$ and $\mathcal{F}(2k, k+1)$ \footnote{It is trivial that $\mathcal{F}(2k,1)=1$, on the other hand, using the series expansion definition of $f_{k,j}$ given in \cite{Francisco} (Definition 2.2),  it can be shown directly that the right hand of \pref{conjforG} is 1 when $j=1$, showing that formula \pref{conjforG} holds for $j=1$. With less ease, but with elementary combinatorial methods it can be shown that $\mathcal{F}(2k, k+1) = \frac{1}{k+1}{{2k}\choose{k}}$ and that $\mathcal{F}(2k, 2) = {{2k}\choose{k}}-1$. The latter equation, together with equation \pref{formulapochhammer} yields that when $N=2$ the $k$-th momment of the ESD of the squarred unimodular matrix is $2^{-2k-1}{{2k}\choose{k}}$, confirming that in this case, the ESD is the arcsine distribution supported in $[0,2]$.}. 

In this article we append two tables; one of them presenting the values of $\mathcal{F}(2k, j)$ for all $1\leq k \leq 11$ and all $1\leq j \leq k+1$, the other table presents the respective values of the right hand side of \pref{conjforG}.  We believe that the right hand side of \pref{conjforG} is actually counting a proper subfamily of the partitions whose associated quotient is a d.d.c.g., this is supported by the fact that the proposed formula has been always seen to be less or equal than $\mathcal{F}(2k, j)$. When there are few blocks, say one or two, or when they are too many, say $k+1$ or $k$, the partitions whose associated quotient is a d.d.c.g. have a very ordered structure. On the other hand, when there are more than two blocks, but not too many, and we have enough vertices, i.e. $k$ is big, the partitions that satisfy the wanted property have no particular structure at all. We think that it is in this case when ``rare" partitions appear that are not counted by the proposed formula. 

Below we show the explicit coefficients in the formula \pref{formulafinal} when $k=6$ and $k=7$.
$$N^{13}\mathbb{E}[\mathrm{tr}(\rho_N^6)] = 132 N^{7}-495 N^{6}+772 N^5- 624 N^4 +262 N^3 - 46 N^2.$$
$$N^{15}\mathbb{E}[\mathrm{tr}(\rho_N^7)]= 429 N^8 - 2002 N^7 + 4039 N^6 -4550 N^5 + 3073 N^4 -1204 N^3 + 216 N^2.$$

Finally, we warn the reader that several attempts were made, using the On-line Encyclopedia of Integer Sequences, to find explicit formulas for the numbers $\mathcal{F}(2k, j)$ with $j$ different from 1, 2, $k$ and $k+1$, and to find formulas for the   coefficients given in Theorem \ref{theorem} for the monomials $N^j$ with $j$ different from 0, 1, $k$ and $k+1$ . It was not even possible to arrive to a conjecture, making us believe that there is no ``nice" general formula for the numbers $\mathcal{F}(2k,j)$ nor for the moments of the mean ESD of $\rho_N$.

\section*{Acknowledgements}

I thank Octavio Arizmendi for introducing me to traffic-free probability and for his valuable comments on this work. I also thank Jorge Fernandez for providing the code to do the computer calculations of the formulas above. Finally, I thank Daniel Perales for providing a nice and simple proof for the equality $\mathcal{F}(2k, 2) = {{2k}\choose{k}}-1$.

\pagebreak

 \begin{center}
 \begin{Table}
 \resizebox{\textwidth}{!}{
\begin{tabular}{ | {c} | *{20}{| c  } }
\hline
 & $2k=2$ & $2k=4$ & $2k=6$  & $2k=8$ & $2k=10$ & $2k=12$ & $2k=14$ & $2k=16$ & $2k=18$ & $2k=20$ & $2k=22$ \\
 \hline 
 \hline
  $j=1$ & 1 & 1  & 1 & 1 & 1 & 1 & 1  & 1   & 1  & 1 & 1               \\
 $j= 2$ & 1 & 5 &  19 & 69 & 251 & 923 &  3431  & 12869 & 48619  & 184755 & 705431      \\
 $j=3$  &  &  2  & 24  & 202 &  1520  & 11000   &  78806  & 566234 & 4105320 & 30114712 & 223707242    \\
 $j=4$  & & & 5 & 112  & 1665   & 21121   &    249137 & 2840928 & 31954529 & 358556005 & 4040139741         \\ 
 $j=5$  & & & & 14  & 510  & 11827  &   226205  & 3918842     & 64318998 &  1025094615 & 16099942903           \\
 $j=6$  & & & & & 42 & 2277 &   76111 & 2044444  & 48721602 & 1081809409  & 23011155057           \\
 $j=7$ & & & & & & 132  & 10010  & 456456 &  16387776 & 513317334  & 14774891956               \\
 $j=8$ & & & & & & &      429     & 43472   & 2596596 & 120110865 & 4781025480                  \\
 $j=9$ & & & & & & & & 1430  & 186966  & 14177490    & 821327364                  \\ 
 $j=10$ & & & & & & & & &  4862 & 797810 & 74918558 \\ 
 $j=11$ & & & & & & & & & & 16796 & 3382456  \\
 $j=12$ & & & & & & & & & & &  58786  \\
\hline 
\end{tabular}}
\caption{Values of $\mathcal{F}(2k, j)$.}
\end{Table}
 \end{center}
 \bigskip
 
 \begin{center}
 \begin{Table}
  \resizebox{\textwidth}{!}{
 \begin{tabular}{| {c} |  *{20}{| c}}
 \hline
 & $2k=2$ & $2k=4$ & $2k=6$  & $2k=8$ & $2k=10$ & $2k=12$ & $2k=14$ & $2k=16$ & $2k=18$ & $2k=20$ & $2k=22$ \\
 \hline 
 \hline
 $j=1$ & 1 & 1  & 1 & 1 & 1 & 1 & 1  & 1   & 1  & 1 & 1                   \\
 $j=2$ &  1 & 5 & 19 & 69  & 251 & 923 & 3431 & 12869 & 48619  & 184755 & 705431               			  \\
 $j=3$ & & 2 & 24 & 202 & 1520 &  10988 & 78428 & 559130	 & 4001136	 & 28795012 & 208515164					 \\
 $j=4$ & & & 5 & 112 & 1665 & 21109 & 248339 & 2813712 & 31278521	 & 344578585 & 3783013707				\\
 $j=5$ & 	&	&	& 14	 & 510	& 	11825	& 225862 & 3896970 & 63425538 & 996691265 & 15328496106			\\
 $j=6$ & 	& & & & 42 & 2277	& 76076	& 2039128	& 48338310 & 1062780789 & 22255811424					\\
 $j=7$ & & & & & & 132 & 10010	& 456092 & 16327752	& 508232748	& 14469523530								\\
 $j=8$ & 	& & & & & & 429	& 43472	& 2593656 & 119555220	& 4725337221													\\
 $j=9$ & 	& & & & & & & 1430 & 186966 & 14157090 & 816841806																\\
$j=10$ & & & & & & & & & 4862 & 797810 	& 74790650																							\\
$j=11$ & & & & & & & & & & 16796 & 3382456																								\\
$j=12$ & & & & & & & & & & & 58786																								\\
\hline
 \end{tabular}}
 \caption{Conjectured values from equation \pref{conjforG}.}
\end{Table}
 \end{center}

\end{document}